\documentclass[10pt]{amsart}
\usepackage[utf8]{inputenc}
\usepackage{setspace}
\usepackage{setspace}
\usepackage[margin=1in]{geometry}
\usepackage[english]{babel}
\usepackage{amsthm}
\usepackage{amsmath}
\usepackage{amsfonts}
\usepackage{amssymb}
\usepackage{amsthm, thmtools}
\usepackage{
hyperref,
cleveref,
}
\declaretheorem[name=Theorem,
refname={Theorem,Theorems},
Refname={Theorem,Theorems},
numberwithin=section]{theorem}
\declaretheorem[name=Lemma,
refname={Lemma,Lemmas},
Refname={Lemma,Lemmas},
sibling=theorem,
style=definition]{lemma}
\declaretheorem[name=Corollary,
refname={Corollary,Corollaries},
Refname={Corollary,Corollaries},
sibling=theorem,
style=definition]{col}
\declaretheorem[name=Example,
refname={Example,Examples},
Refname={Example,Examples},
sibling=theorem,
style=definition]{eg}
\declaretheorem[name=Proposition,
refname={Proposition,Propositions},
Refname={Proposition,Propositions},
sibling=theorem,
style=definition]{propn}
\declaretheorem[name=Definition,
refname={Definition,Definitions},
Refname={Definition,Definitions},
sibling=theorem,
style=definition]{defn}

\declaretheorem[name=Remark,
refname={Remark,Remarks},
Refname={Remark,Remarks},
sibling=theorem]{rmk}
\declaretheorem[name=Notation,
refname={Notation,Notations},
Refname={Notation,Notations},
sibling=theorem]{notation}
\newcommand{\setR}{\ensuremath{\mathbb{R}}}

\newcommand{\setN}{\ensuremath{\mathbb{N}}}

\newcommand{\setZ}{\ensuremath{\mathbb{Z}}}

\newcommand{\setQ}{\ensuremath{\mathbb{Q}}}

\newcommand{\R}{\mathbb{R}}
\newcommand{\sR}{\mathcal{R}}
\newcommand{\lM}{\mathcal{M}_L}
\newcommand{\N}{\mathbb{N}}

\newcommand{\Q}{\mathbb{Q}}

\newcommand{\bkl}{\text{ \textbackslash \ }}

\title{On a new measure on the Levi-Civita field $\sR$}
\author{Mateo Restrepo Borrero$^*$}
\address{Department of Mathematics, Universidad Nacional de Colombia, Colombia}
\email{marestrepob@unal.edu.co}
\author{Vatsal Srivastava$^*$}
\address{Department of Mathematics, Indian Institute of Technology Bombay, Mumbai, India}
\email{vatsal.sri.math@gmail.com}
\thanks{$^*$ Undergraduate student supported by the MITACS Globalink program}
\author{Khodr Shamseddine$^\dagger$}
\address{Department of Physics and Astronomy and Department of Mathematics, University of Manitoba, Winnipeg, Manitoba
R3T 2N2, Canada}
\email{khodr.shamseddine@umanitoba.ca}
\thanks{$^\dagger$ Research funded by the Natural Sciences and Engineering Council of Canada (NSERC, Grant \# RGPIN/4965-2017)}
\date{}

\begin{document}

\maketitle
\markboth{M. Restrepo Borrero, V. Srivastava, and K. Shamseddine}{On a new measure on the Levi-Civita field $\sR$}

\begin{abstract}
   The Levi-Civita field $\sR$ is the smallest non-Archimidean ordered field extension of the real numbers that is real closed and Cauchy complete in the topology induced by the order. In an earlier paper \cite{RSMNi02}, a measure was defined on $\sR$ in terms of the limit of the sums of the lengths of inner and outer covers of a set by countable unions of intervals as those inner and outer sums get closer together. That definition proved useful in developing an integration theory over $\sR$ in which the integral satisfies many of the essential properties of the Lebesgue integral of real analysis. Nevertheless, that measure theory lacks some intuitive results that one would expect in any reasonable definition for a measure; for example, the complement of a measurable set within another measurable set need not be measurable.

In this paper, we will give a characterization for the measurable sets defined in \cite{RSMNi02}. Then we will introduce the notion of an outer measure on $\sR$ and show some key properties the outer measure has. Finally, we will use the notion of outer measure to define a new measure on $\sR$ that proves to be a better generalization of the Lebesgue measure from $\R$ to $\sR$ and that leads to a family of measurable sets in $\sR$ that strictly contains the family of measurable sets from \cite{RSMNi02}, and for which most of the classic results for Lebesgue measurable sets in $\R$ hold.
\end{abstract}
\section{Introduction}

We recall that the elements of the Levi-Civita field $\mathcal{R}$ are functions
from $\ensuremath{\mathbb{Q}}$ to $\ensuremath{\mathbb{R}}$ with left-finite support (denoted by
supp). That is, for every $q\in\setQ$ there are only finitely many elements in the support that are smaller than $q$. For the further
discussion, it is convenient to introduce the following terminology.

\begin{defn}($\lambda $, $\sim $, $\approx $)
We define $\lambda: \mathcal{R}\rightarrow\setQ$ by
\[
\lambda(x)=\left\{\begin{array}{ll}
\min (\mbox{supp}(x))&\text{if }x\ne0\\
\infty&\text{if }x=0.
\end{array}\right.
\]
The minimum exists because of the
left-finiteness of supp$(x)$. Moreover, we denote the value of $x$ at $q\in\setQ$ with brackets like $x[q]$.

Given $x,y\ne 0$ in $\mathcal{R}$, we say $x\sim y$ if $\lambda (x)=\lambda (y)$; and we say $x\approx y$ if $\lambda
(x)=\lambda (y)$ and $x[\lambda (x)]=y[\lambda (y)] $.
\end{defn}

At this point, these definitions may feel somewhat arbitrary; but after having introduced an order on $\mathcal{R}$, we will see that $\lambda $
describes orders of magnitude, $\sim$ corresponds
to agreement of the order of magnitude, while $\approx$ corresponds to agreement up to infinitely small relative error.

The set $\mathcal{R}$ is endowed with formal
power series multiplication and componentwise addition, which make
it into a field \cite{shamseddinephd} in which we can isomorphically
embed the field of real numbers $\setR$  as a subfield via the map
$E :\setR\rightarrow \mathcal{R}$ defined by
\begin{equation}
E (x)[q]=\left\{
\begin{array}{ll}
x & \mbox{ if }q=0 \\
0 & \mbox{ else.}
\end{array}
\right.  \label{eq:embed}
\end{equation}

\begin{defn} (Order in $\mathcal{R}$)
Let $x,y\in\mathcal{R}$ be given. Then we say that $x>y$ (or $y<x$) if $x\ne y$ and
$(x-y)[\lambda (x-y)]>0$; and we say $x\ge y$ (or $y\le x$) if $x=y$ or $x>y$.
\end{defn}

It follows that the relation $\ge$ (or $\le$) defines a total order on $\mathcal{R}$ which makes it into an ordered field.
Note that, given $a<b$ in $\mathcal{R}$, we define the $\mathcal{R}$-interval
$[a,b] = \{x\in\mathcal{R}:a\le x\le b\}$, with the obvious adjustments in the definitions of the intervals $[a,b)$, $(a,b]$, and $(a,b)$.
Moreover, the embedding $E$ in Equation (\ref{eq:embed}) of
$\setR$ into $\mathcal{R}$ is compatible with the order.

The order leads to the definition of an ordinary absolute value on $\mathcal{R}$:
\[
 |x|=\max\{x,-x\}=\left\{\begin{array}{ll}x&\mbox{if }x\ge 0\\
 -x&\mbox{if }x<0;
 \end{array}\right.
 \]
 which induces the same topology on $\mathcal{R}$ (called the order topology or valuation topology) as that induced by the ultrametric absolute value $|\cdot|_u:\mathcal{R}\rightarrow\setR$, given by
 \[
  |x|_u=\left\{\begin{array}{ll}
  e^{-\lambda(x)}&\text{if }x\ne0\\
  0&\text{if }x=0,
  \end{array}\right.
  \]
  as was shown in \cite{rspascal05}.

We note in passing here that $\left\vert
\cdot\right\vert_u$ is a non-Archimedean valuation on $\mathcal{R}$; that is, it satisfies the following properties
\begin{enumerate}
\item $|v|_u\ge 0$ for all $v\in\mathcal{R}$ and $|v|_u=0$ if and only if $v=0$;
\item $|vw|_u=|v|_u|w|_u$ for all $v,w\in\mathcal{R}$; and
\item $|v+w|_u\le \max\{|v|_u,|w|_u\}$ for all $v,w\in\mathcal{R}$: the strong triangle inequality.
\end{enumerate}
Thus, $(\mathcal{R},|\cdot|_u)$ is a non-Archimedean valued field. Moreover, $|.|_u$ induces a metric $\Delta$ on $\mathcal{R}$ given by $\Delta(x,y)=|y-x|_u$ which satisfies the strong triangle inequality and is thus an ultrametric, making $(\sR,\Delta)$ an utrametric space.

Besides the usual order relations on $\mathcal{R}$, some other notations are
convenient.

\begin{defn} ($\ll ,\gg $)
Let $x,y\in\mathcal{R}$ be non-negative. We say $x$ is infinitely
smaller than $y$ (and write $x\ll y$) if $nx<y$ for all $n\in\setN$;
we say $x$ is infinitely larger than $y$ (and write $x\gg y$) if
$y\ll x$. If $x\ll 1,$ we say $x$ is infinitely small; if $x\gg 1$,
we say $x$ is infinitely large. Infinitely small numbers are also
called infinitesimals or differentials. Infinitely large numbers are
also called infinite. Non-negative numbers that are neither
infinitely small nor infinitely large are also called finite.
\end{defn}

\begin{defn}(The Number $d$)
\label{def:d}Let $d$ be the element of $\mathcal{R}$ given by $d[1]=1$ and $%
d[t]=0$ for $t\ne 1$.
\end{defn}

\begin{rmk}
Given $m\in\setZ$, then $d^m$ is the positive $\mathcal{R}$-number given by
\[
d^m=\left\{\begin{array}{ll}\underbrace{dd\cdots d}_{m \mbox{ times}}&\mbox{ if }m>0\\
\mbox{}&\mbox{}\\
1&\mbox{ if }m=0\\
\mbox{}&\mbox{}\\
\frac{1}{d^{-m}}&\mbox{ if }m<0
\end{array}\right. .
\]
Moreover, given a rational number $q=m/n$ (with $n\in\setN$ and $m\in \setZ$), then $d^q$ is  the positive $n$th root of $d^m$ in $\mathcal{R}$ (that is, $(d^q)^n=d^m$) and it is given by
\[
d^{q}[t]= \left\{
\begin{array}{ll}
1& \mbox{ if }t=q \\
0 & \mbox{otherwise.}
\end{array}
\right.
\]

It is easy to check that $d^q\ll 1$ if $q>0$ and $d^q\gg 1$ if
$q<0$ in $\setQ$. Moreover, for all $x\in \mathcal{R}$,
the elements of supp$(x)$ can be arranged in ascending order, say
supp$(x)=\{q_1, q_2, \ldots\}$ with $q_j<q_{j+1}$ for all $j$; and
$x$ can be written as $x=\sum\limits_{j=1}^\infty x[q_j]d^{q_j}$, where the
series converges in the order (valuation) topology \cite{rscrptsf}.
\end{rmk}

Altogether, it follows that $\mathcal{R}$ is a non-Archimedean (valued and ordered) field extension of $\setR$. For a detailed study of this field, we refer the
reader to the survey paper \cite{rsrevitaly13} and the references therein. In particular, it is
shown that $\mathcal{R}$ is complete with respect to the natural (valuation) topology or, equivalently, with respect to the ultrametric $\Delta$.

It follows therefore that $\mathcal{R}$
is just a special case of the class of fields discussed in
\cite{schikhofbook}. For a general
overview of the algebraic properties of formal power series fields, we refer to the comprehensive overview by Ribenboim \cite{ribenboim92}, and
for an overview of the related valuation theory, to the book by Krull \cite{krull32}. A thorough and complete treatment of ordered structures can
also be found in \cite{priessbook}. A more comprehensive survey of all non-Archimedean fields can be found in \cite{barria-sham-18}.

Besides being the smallest non-Archimedean ordered field extension
of the real numbers that is both complete in the order topology and
real closed, the Levi-Civita field $\mathcal{R}$ is of particular
interest because of its practical usefulness. Because of the left-finiteness of the supports of the Levi-Civita numbers, those numbers can be used on a computer, thus allowing for many useful computational applications.
One such application is the computation of derivatives
of real functions representable on a computer
\cite{rsdiffsf}, where both the accuracy of formula
manipulators and the speed of classical numerical methods are
achieved.

The following result is not special to $\mathcal{R}$ but it holds in any non-Archimedean valued field; its proof can be found in \cite{shamseddinephd,rspsio00}.
\begin{propn} Let $\{a_n\}_{n\in\N}$ be a sequence in $\mathcal{R}$. Then $\{ a_n\}$ is a Cauchy sequence in the valuation topology if and only if $\lim\limits_{n\to\infty}\left(a_{n+1}-a_n\right)=0$.
\end{propn}

Since $\mathcal{R}$ is Cauchy complete, we readily obtain the following result.
\begin{col}\label{corseqconv} Let $\{a_n\}_{n\in\N}$ be a sequence in $\mathcal{R}$. Then $\{ a_n\}$ converges in $\mathcal{R}$
 if and only if $\lim\limits_{n\to\infty}\left(a_{n+1}-a_n\right)=0$.
\end{col}
\begin{col}\label{corinfinitesum} Let $\{a_n\}_{n\in\N}$ be a sequence in $\mathcal{R}$. Then $\sum\limits_{n\in\N}a_n$ converges in $\mathcal{R}$
 if and only if $\lim\limits_{n\to\infty}a_n=0$.
\end{col}

Moreover, thanks to the non-Archimedean (ultrametric) nature of $\mathcal{R}$, the order of limits, including double infinite sums, can be interchanged more conveniently than in $\R$ \cite{shamseddinephd}.

\section{The S-Measure on $\mathcal{R}$}\label{secMI1D} Using the nice smoothness properties of power series (see \cite{rspssurv13} and the references therein), we developed a measure and integration theory on $\mathcal{R}$ in \cite{RSMNi02,rsint2} that uses the $\mathcal{R}$-analytic functions (functions given locally by power series) as the building blocks for measurable functions instead of the step functions used in the real case. We will refer to that measure by the S-measure henceforth in this paper.

\begin{notation} Let $a<b$ in $\mathcal{R}$ be given. Then by $l(I(a,b))$ we will denote the length of the interval $I(a,b)$, that is
\[
l(I(a,b))=\mbox{ length of }I(a,b)= b-a.
\]
\end{notation}

\begin{defn}\label{S-measureDef}
Let $A\subset\mathcal{R}$ be given. Then we say that $A$ is
S-measurable if for every $\epsilon>0$ in $\mathcal{R}$, there exist a sequence
of pairwise disjoint intervals $(I_{n})$ and a sequence of pairwise disjoint
intervals $(J_{n})$ such that $\bigcup\limits_{n=1}^{\infty}I_{n}\subset A\subset
\bigcup\limits_{n=1}^{\infty}J_{n}$, $\sum\limits_{n=1}^{\infty}l(I_{n})$ and $\sum
_{n=1}^{\infty}l(J_{n})$ converge in $\mathcal{R}$, and $\sum\limits_{n=1}^{\infty
}l(J_{n})-\sum\limits_{n=1}^{\infty}l(I_{n})\leq\epsilon$.
\end{defn}

Given an S-measurable set $A$, then for every
$k\in\ensuremath{\mathbb{N}}$, we
can select a sequence of pairwise disjoint intervals $\left(  I_{n}%
^{k}\right)  $ and a sequence of pairwise disjoint intervals $\left(
J_{n}^{k}\right)  $ such that $\sum\limits_{n=1}^{\infty}l\left(
I_{n}^{k}\right)  $ and $\sum\limits_{n=1}^{\infty}l\left(
J_{n}^{k}\right)  $ converge in $\mathcal{R}$ for all $k$,
\[
\bigcup\limits_{n=1}^{\infty}I_{n}^{k}\subset\bigcup\limits_{n=1}^{\infty}I_{n}^{k+1}\subset
A\subset\bigcup\limits_{n=1}^{\infty}J_{n}^{k+1}\subset\bigcup\limits_{n=1}^{\infty}J_{n}^{k}
\mbox{ and }\sum\limits_{n=1}^{\infty}l\left(  J_{n}^{k}\right)
-\sum\limits_{n=1}^{\infty }l\left(  I_{n}^{k}\right)  \le d^{k}
\]
for all $k\in\setN$. Since $\mathcal{R}$ is Cauchy complete in the
order (valuation) topology, it follows that $\lim\limits_{k\rightarrow\infty}\sum
_{n=1}^{\infty}l\left(  I_{n}^{k}\right)  $ and $\lim\limits_{k\rightarrow\infty}%
\sum\limits_{n=1}^{\infty}l\left(  J_{n}^{k}\right)  $ both exist and they
are equal. We call the common value of the limits the S-measure of $A$
and we denote it by $M_s(A)$. Thus,
\[
M_s(A)=\lim\limits_{k\rightarrow\infty}\sum\limits_{n=1}^{\infty}l\left(
I_{n}^{k}\right)  =
\lim\limits_{k\rightarrow\infty}\sum\limits_{n=1}^{\infty}l\left(
J_{n}^{k}\right)  .
\]
Contrary to the real case,
\[
\sup\left\{  \sum\limits_{n=1}^{\infty}l(I_{n}):I_{n}\mbox{'s are pairwise disjoint intervals}  \mbox{ and
}  \bigcup\limits_{n=1}^{\infty
}I_{n}\subset A\right\}
\]
 and
 \[
 \inf\left\{  \sum\limits_{n=1}^{\infty}l(J_{n}%
):J_{n}\mbox{'s are pairwise disjoint intervals}  \mbox{ and
} A\subset\bigcup\limits_{n=1}^{\infty}J_{n}\right\}
\]
need not exist for a
given set $A\subset\mathcal{R}$. However, as shown in \cite{RSMNi02}, if
$A$ is S-measurable then both the supremum and infimum exist
and they are equal to $M_s(A)$. This shows that the definition of S-measurable
sets in Definition \ref{S-measureDef} is a good generalization of that of the Lebesgue
measurable sets of real analysis that corrects for the lack of suprema and infima in
non-Archimedean ordered fields.

It follows directly from the definition
that $M_s(A)\ge0$ for any S-measurable set $A\subset\mathcal{R}$ and that any
interval $I(a,b)$ is S-measurable with S-measure $M_s(I(a,b))=l(I(a,b))=b-a$. It also follows
that if $A$ is a countable union of pairwise disjoint intervals $\left(
I_{n}(a_{n},b_{n})\right)  $ such that $\sum\limits_{n=1}^{\infty}(b_{n}-a_{n})$
converges then $A$ is S-measurable with $M_s(A)= \sum\limits_{n=1}^{\infty}(b_{n}-a_{n}%
)$. Moreover, if $B\subset A\subset\mathcal{R}$ and if $A$ and $B$ are
S-measurable, then $M_s(B)\le M_s(A)$.

In \cite{RSMNi02} we show that the S-measure defined on $\mathcal{R}$ above has similar
properties to those of the Lebesgue measure on $\ensuremath{\mathbb{R}}$. For example,
we show that any subset of an S-measurable
set of S-measure $0$ is itself S-measurable and has S-measure $0$. We also show that
any countable unions of S-measurable sets whose S-measures form a null sequence is
S-measurable and the S-measure of the union is less than or equal to the sum of
the S-measures of the original sets; moreover, the S-measure of the union is equal
to the sum of the S-measures of the original sets if the latter are pairwise
disjoint. Furthermore, we show that any finite intersection of S-measurable sets is also
S-measurable and that the sum of the S-measures of two S-measurable sets is equal to
the sum of the S-measures of their union and intersection.

It is worth noting that the complement of an S-measurable set in an S-measurable set
need not be S-measurable. For example, $[0,1]$ and $[0,1]\cap\ensuremath
{\mathbb{Q}}$ are both S-measurable with S-measures $1$ and $0$, respectively.
However, the complement of $[0,1]\cap\ensuremath{\mathbb{Q}}$ in $[0,1]$ is
not S-measurable. On the other hand, if $B\subset A\subset\mathcal{R}$ and if
$A$, $B$ and $A\setminus B$ are all S-measurable, then $M_s(A)=M_s(B)+M_s(A\setminus B)$.

The example of $[0,1]\setminus[0,1]\cap\ensuremath{\mathbb{Q}}$ above shows
that the axiom of choice is not needed here to construct a set that is not S-measurable,
as there are many simple examples of such sets. Indeed, any
uncountable real subset of $\mathcal{R}$, like $[0,1]\cap\ensuremath
{\mathbb{R}}$ for example, is not S-measurable. This ease of finding subsets of $\mathcal{R}$ that are not S-measurable may
 seem surprising; however, through closer inspection and the following characterization (Theorem \ref{thmcharac}), it will become obvious that the family of S-measurable sets is simply too narrow, thus the need for a new measure on $\sR$ that will extend the family of S-measurable sets and will share more of the nice properties of the Lebesgue measure on $\R$.

\begin{theorem}\label{thmcharac}
    Let $A\subset\mathcal{R}$ be S-measurable. Then $A$ can be written as a disjoint union $A=\left(\bigcup\limits_{n=1}^\infty K_n\right)\cup S$, where $K_n$ is an interval in $\mathcal{R}$ for each $n\in\N$ and where $\sum\limits_{n=1}^\infty l(K_n)=M_s(A)$ and $M_s(S)=0$.
\end{theorem}
\begin{proof}
    Let $\epsilon>0$ in $\mathcal{R}$ be given. By definition, there exist two sequences of pairwise disjoint intervals $\{I_n\}$ and $\{J_n\}$ such that ${\bigcup\limits_{n=1}^\infty I_n\subseteq A\subseteq \bigcup\limits_{n=1}^\infty J_n}$,  $\sum\limits_{n=1}^\infty l(I_n)$ and $\sum\limits_{n=1}^\infty l(J_n)$ both converge in the order topology, and $\sum\limits_{n=1}^\infty l(J_n)-\sum\limits_{n=1}^\infty l(I_n)<\epsilon/2$.

We can re-write the collection $\{I_n\}_{n=1}^\infty$ as $\bigcup\limits_{m=1}^\infty\{I_n\cap J_m\}_{n=1}^\infty$. Since, for every $m\in\setN$, we have that $\lim\limits_{n\rightarrow\infty}l(I_n\cap J_m)=0$, it follows that $\sum\limits_{n=1}^\infty l(I_n\cap J_m)$ converges for every $m\in\setN$, by Corollary \ref{corinfinitesum}. Thus, there exists $N_m\in\setN$ such that $\sum\limits_{n=N_m+1}^\infty l(I_n\cap J_m)<d^m\epsilon$. It follows that
\begin{align*}
    \sum\limits_{n=1}^\infty l(J_n)-\sum\limits_{m=1}^\infty\sum\limits_{n=1}^{N_m} l(I_n\cap J_m)&\leq \sum\limits_{n=1}^\infty l(J_n)-\sum\limits_{m=1}^\infty\left[\sum\limits_{n=1}^\infty l(I_n\cap J_m)-d^m\epsilon\right]\\
    &=\sum\limits_{n=1}^\infty l(J_n)-\sum\limits_{n=1}^\infty\sum\limits_{m=1}^\infty l(I_n\cap J_m)+\sum\limits_{m=1}^\infty d^m\epsilon\\
     &=\sum\limits_{n=1}^\infty l(J_n)-\sum\limits_{n=1}^\infty l(I_n)+\sum\limits_{m=1}^\infty d^m\epsilon\\
     &<\frac{\epsilon}2+\frac{d}{1-d}\epsilon\\
     &<\epsilon.
\end{align*}
Thus, we can replace the original collections of intervals $\{I_n\}$ and $\{J_n\}$ with $\bigcup\limits_{m=1}^\infty \{J_m\cap I_n\}_{n=1}^{N_m}$ and $\{J_n\}$ which can be easily re-written as $\{S_n\}$, $\{X_n\}$ where $S_n\subseteq X_n$ for each $n$. Moreover, since $X_n\bkl S_n$ is at most the disjoint union of two intervals, we can write $\{X_n\}=\{S_n\}\cup\{R_n\}$ where $\sum\limits_{n=1}^\infty l(R_n)<\epsilon$.

Now, take $\epsilon = d$. As shown, we can find two sequences of pairwise disjoint intervals $\{S_n^1\}$ and $\{R_n^1\}$ such that
$$\bigcup\limits_{n=1}^\infty S_n^1\subseteq A \subseteq \left(\bigcup\limits_{n=1}^\infty S_n^1\right)\cup \left(\bigcup\limits_{n=1}^\infty R_n^1\right)\text{ and }\sum\limits_{n=1}^\infty l(R_n^1)<d.$$
Now, given an arbitrary $k\in\setN$, assume that for every positive integer $m\leq k$ we have a pair of sequences of pairwise disjoint intervals $\{S_n^m\}$ and $\{R_n^m\}$ such that $\bigcup\limits_{n=1}^\infty S_n^m\subseteq A \subseteq \left(\bigcup\limits_{n=1}^\infty S_n^m\right)\cup \left(\bigcup\limits_{n=1}^\infty R_n^m\right)$, $\sum\limits_{n=1}^\infty l(R_n^m)<d^m$, and $\{S_n^m\}\subseteq \{S_n^{m+1}\}$. Take now a pair of sequences of pairwise disjoint intervals $\{I_n\}$ and $\{O_n\}$ such that
$$\bigcup\limits_{n=1}^\infty I_n\subseteq A \subseteq \left(\bigcup\limits_{n=1}^\infty I_n\right)\cup \left(\bigcup\limits_{n=1}^\infty O_n\right)\text{ and }\sum\limits_{n=1}^\infty l(O_n)<d^{k+1}.$$
Consider the collections of pairwise disjoint intervals $\bigcup\limits_{m=1}^\infty\{I_n\cap R_m^k\}$ and $\bigcup\limits_{m=1}^\infty\{O_n\cap R_m^k\}$. We define
$$\{R_n^{k+1}\}:=\bigcup\limits_{m=1}^\infty\{O_n\cap R_m^k\}\text{ and }\{S_n^{k+1}\}:=\{S_n^k\}\cup\left(\bigcup\limits_{m=1}^\infty\{I_n\cap R_m^k\}\right).$$
Then $\{S_n^{k+1}\}$ and $\{R_n^{k+1}\}$ are pairwise disjoint collections of intervals that satisfy
$$\bigcup\limits_{n=1}^\infty S_n^{k+1}\subseteq A \subseteq \left(\bigcup\limits_{n=1}^\infty S_n^{k+1}\right)\cup \left(\bigcup\limits_{n=1}^\infty R_n^{k+1}\right)$$
and
\[
    \sum\limits_{n=1}^\infty l(R_n^{k+1})=\sum\limits_{n=1}^\infty\sum\limits_{m=1}^\infty l(O_n\cap R_m^k)\leq \sum\limits_{n=1}^\infty l(O_n)<d^{k+1}.
\]
We define $\{S_n^\infty\}=\bigcup\limits_{k=1}^\infty\{S_n^k\}$, which is a disjoint countable union of intervals that are contained in $A$. It follows that $\{R_n^k\}$ is a sequence of covers of $A\bkl \bigcup\limits_{n=1}^\infty S_n^\infty$ that satisfies the condition $\lim\limits_{k\to\infty}\sum\limits_{n=1}^\infty l(R_n^k)=0$. Thus,
$$\sum\limits_{n=1}^\infty l(S_n^\infty)\leq M_s(A)\leq\sum\limits_{n=1}^\infty l(S_n^\infty)+\lim\limits_{k\to\infty}\sum\limits_{n=1}^\infty l(R_n^k)=\sum\limits_{n=1}^\infty l(S_n^\infty)$$
We conclude that $A=\left(\bigcup\limits_{n=1}^\infty K_n\right)\cup S$ where $\sum\limits_{n=1}^\infty l(K_n)=M_s(A)$ and $M_s(S)=0$.
\end{proof}
\section{The outer measure}
The effect of having too small a family of S-measurable sets impedes further progress into more significant results that the reader associates with the Lebesgue measure in $\R$. So we will introduce a new definition that will enlarge the pool of measurable sets while still circumventing the fact that not all bounded sets in $\sR$ have an infimum or a supremum.
\begin{defn}
    Let $A\subset\sR$ be given. Then we say that $A$ is outer measurable if
    $$\inf \left\{\sum\limits_{n=1}^\infty l(S_n): S_n\text{'s are intervals and }A\subseteq \bigcup\limits_{n=1}^\infty S_n \right\}$$
    exists in $\sR$. If so, we call that number the outer measure of $A$ and denote it by $M_u(A)$.
\end{defn}
\subsection{General properties}
As we shall see, this definition is going to assist us in defining a larger family of measurable sets than that of the S-measurable sets. Unfortunately, working with the infimum of a set is often difficult and labour intensive. To avoid that, we will make use of a variant of the assertion made immediately after Definition \ref{S-measureDef}, which was used to define the S-measure of an S-measurable set.
\begin{propn}
Let $A\subset\sR$ be outer measurable. Then there exists a sequence of sequences of pairwise disjoint intervals $\left(\{I_n^k\}_{n=1}^\infty\right)_{k=1}^\infty$ such that $\lim\limits_{k\to\infty}\sum\limits_{n=1}^\infty l(I_n^k)=M_u(A)$, and for all $k\in\N$, we have that
$$A\subseteq\bigcup\limits_{n=1}^\infty I_n^{k+1}\subseteq\bigcup\limits_{n=1}^\infty I_n^k.$$
We say that such a sequence converges to $A$.
\end{propn}

\begin{lemma}
    Let $A$, $B$ and $C$ be outer measurable sets in $\mathcal{R}$ such that $A\subseteq B\cup C$. Then $M_u(A)\leq M_u(B)+M_u(C)$.
\end{lemma}
\begin{proof}
    Let $\{I_n\}$, $\{J_n\}$ be arbitrary covers of $B$ and $C$, respectively. Then, $\{I_n\}\cup\{J_n\}$ is a cover of $A$ and, by definition,
    $$M_u(A)\leq \sum\limits_{n=1}^\infty l(I_n)+\sum\limits_{n=1}^\infty l(J_n)$$
    It follows that
    $$M_u(A)-\sum\limits_{n=1}^\infty l(I_n)\leq \sum\limits_{n=1}^\infty l(J_n)$$
    Thus, $M_u(A)-\sum\limits_{n=1}^\infty l(I_n)$ is a lower bound for the set $X_C:=\left\{\sum\limits_{n=1}^\infty l(S_n):C\subseteq \bigcup\limits_{n=1}^\infty S_n \right\}$ and hence
     $$M_u(A)-\sum\limits_{n=1}^\infty l(I_n)\leq \inf(X_C)=M_u(C).$$
     Thus,
     $$M_u(A)-M_u(C)\leq \sum\limits_{n=1}^\infty l(I_n)$$
     and hence
     $$M_u(A)-M_u(C)\leq \inf(X_B)=M_u(B),$$
     from which the result follows.
\end{proof}
It turns out that sets of outer measure zero inherit one of the key properties that hold for the classical Lebesgue measurable subsets of $\R$ of measure $0$, as shown in the following proposition.
\begin{propn}\label{0measuresets}
Let $A,B\subset\sR$ be outer measurable with $M_u(B)=0$. Then, for any subset $C\subseteq B$ we have that $M_u(C)=0$ and $M_u(A\bkl C)=M_u(A)$.
\end{propn}
\begin{proof}
It follows immediately from the definition that $M_u(C)=0$. To see that $M_u(A\bkl C)=M_u(A)$, it is enough to notice that if $\{J_n^k\}$ converges to $A$ then it so converges to $A\bkl C$, for if $\{S_n\}$ covers $A\bkl C$ and $\sum\limits_{n=1}^\infty l(S_n)<M_u(A)$, we can find $\{I_n\}$ covering $C$ such that $\sum\limits_{n=1}^\infty l(I_n)+\sum\limits_{n=1}^\infty l(S_n)<M_u(A)$. This will yield a contradiction, since $\{I_n\}\cup\{S_n\}$ covers $A$.
\end{proof}
\subsection{Intervals and the outer measure}
We now introduce a series of results showing that intervals behave particularly well with the notion of outer measure.
\begin{propn}\label{AintersecI_isuppmeasurable}
Let $A\subset\mathcal{R}$ be outer measurable and let $I$ be an interval in $\mathcal{R}$. Then $A\cap I$ is outer measurable.
\end{propn}
\begin{proof}
Let $\left(\{J_n^k\}\right)_{k=1}^\infty$ converge to $A$. For each $n,k\in\setN$, we define
$$I_n^k:=I\cap J_n^k.$$
Then, clearly, $l(I_n^k)\leq l(J_n^k)$ for all $n,k\in\setN$. It follows that, for every $k\in\setN$, the series $\sum\limits_{n=1}^\infty l(I_n^k)$ converges. We will show that $\lim\limits_{k\to\infty}\sum\limits_{n=1}^\infty l(I_n^k)$ exists and is equal to $M_u(A\cap I)$.

First we note that since for every $k$, $\bigcup\limits_{n=1}^\infty (J_n^{k+1}\cap I^c)\subseteq\bigcup\limits_{n=1}^\infty (J_n^k\cap I^c)$ we have that $\sum\limits_{n=1}^\infty l(J_n^{k+1}\cap I^c)\leq\sum\limits_{n=1}^\infty l(J_n^k\cap I^c)$. It follows that
\begin{align*}
   \sum\limits_{n=1}^\infty l(I_n^k)- \sum\limits_{n=1}^\infty l(I_n^{k+1})&=\sum\limits_{n=1}^\infty l(J_n^k\cap I)-\sum\limits_{n=1}^\infty l(J_n^{k+1}\cap I)\\
    &=\sum\limits_{n=1}^\infty (l(J_n^k)-l(J_n^k\cap I^c))-\sum\limits_{n=1}^\infty( l(J_n^{k+1})-l(J_n^{k+1}\cap I^c))\\
    &=\sum\limits_{n=1}^\infty l(J_n^k)- \sum\limits_{n=1}^\infty l(J_n^{k+1})+\sum\limits_{n=1}^\infty l(J_n^{k+1}\cap I^c)-\sum\limits_{n=1}^\infty l(J_n^{k}\cap I^c)\\
    &\leq  \sum\limits_{n=1}^\infty l(J_n^k)- \sum\limits_{n=1}^\infty l(J_n^{k+1})
\end{align*}
Since $0\le\sum\limits_{n=1}^\infty l(I_n^k)- \sum\limits_{n=1}^\infty l(I_n^{k+1})\leq  \sum\limits_{n=1}^\infty l(J_n^k)- \sum\limits_{n=1}^\infty l(J_n^{k+1})$ and since $\lim\limits_{k\to\infty}\left(\sum\limits_{n=1}^\infty l(J_n^k)- \sum\limits_{n=1}^\infty l(J_n^{k+1})\right)=0$, we obtain that $\lim\limits_{
k\to\infty}\left(\sum\limits_{n=1}^\infty l(I_n^k)- \sum\limits_{n=1}^\infty l(I_n^{k+1})\right)=0$. It follows then from Corollary \ref{corseqconv} that $\lim\limits_{k\to\infty}\sum\limits_{n=1}^\infty l(I_n^k)$ exists in $\mathcal{R}$. Let $x:=\lim\limits_{k\to\infty}\sum\limits_{n=1}^\infty l(I_n^k)$.

Suppose now that $(J_n)$ is a cover of $A\cap I$ by pairwise disjoint open intervals. We will show that $\sum\limits_{n=1}^\infty l(J_n)\ge x$. Suppose, to the contrary, that $\sum\limits_{n=1}^\infty l(J_n)<x$. It follows that
\begin{align*}
M_u(A)&=\lim_{k\to\infty}\sum\limits_{n=1}^\infty l(J_n^k)\\
    &=\lim_{k\to\infty}\sum\limits_{n=1}^\infty l(I\cap J_n^k)+\lim_{k\to\infty}\sum\limits_{n=1}^\infty l(I^c\cap J_n^k)\\
    &= x+\lim_{k\to\infty}\sum\limits_{n=1}^\infty l(I^c\cap J_n^k)\\
    &>\sum\limits_{n=1}^\infty l(J_n)+\lim_{k\to\infty}\sum\limits_{n=1}^\infty l(I^c\cap J_n^k)
\end{align*}
It follows that $M_u(A)>\sum\limits_{n=1}^\infty l(J_n)+\sum\limits_{n=1}^\infty l(I^c\cap J_n^k)$ for $k$ large enough. This is a contradiction, since ${\{J_n\}\cup\{I^c\cap J_n^k\}}$ is a cover of $A$ by intervals. We conclude that $A\cap I$ is outer measurable.
\end{proof}
\begin{rmk}
\label{AintersecIcompl_isuppmeasurable}
Using a similar argument, one can prove that $A\cap I^c$ is also outer measurable for any interval $I$ and for any outer measurable set $A$.
\end{rmk}
\begin{col}\label{Acaporminusfiniteintervals}
    Let $A\subset\sR$ be outer measurable and let $\{I_n\}_{n=1}^N$ be a sequence of intervals in $\sR$. Then
    $$A\cap \bigcup\limits_{n=1}^N I_n
   \text{ and }
    A\bkl \bigcup\limits_{n=1}^N I_n$$
    are outer measurable.
\end{col}
\begin{propn}
Let $A\subset\mathcal{R}$ be outer measurable and let $I,J$ be two disjoint intervals in $\mathcal{R}$. Then $(A\cap I)\cup(A\cap J)$ is outer measurable. Moreover,
$$M_u((A\cap I)\cup(A\cap J))=M_u(A\cap I)+M_u(A\cap J).$$
\end{propn}
\begin{proof}
Since $A\cap I$ and $A\cap J$ are outer measurable by Proposition \ref{AintersecI_isuppmeasurable}, there exist two sequences of sequences of intervals $\{I_n^k\}$, $\{J_n^k\}$ converging to $A\cap I$ and $A\cap J$, respectively. Now, the sequence of sequences $\{I_n^k\}\cup\{J_n^k\}$ is a cover of $(A\cap I)\cup(A\cap J)$ satisfying that:
$$\lim_{k\to \infty}\sum\limits_{X\in \{I_n^k\}\cup\{J_n^k\}}l(X)=\lim_{k\to \infty}\left[\sum\limits_{n=1}^\infty l(I_n^k)+\sum\limits_{n=1}^\infty l(J_n^k)\right]=M_u(A\cap I)+M_u(A\cap J).$$
Now let $\{S_n\}$ be a cover of $(A\cap I)\cup(A\cap J)$. Without loss of generality, we may assume that $S_n=S_n\cap I$ or  $S_n=S_n\cap J$. We now may subdivide $\{S_n\}$ into $\{S_n\cap I\}\cup\{S_n\cap J\}:=\{I_n\}\cup\{J_n\}$. It follows that $\{I_n\}$ covers $A\cap I$ and $\{J_n\}$ covers $A\cap J$. Thus,
$$\sum\limits_{n=1}^\infty l(S_n)=\sum\limits_{n=1}^\infty l(I_n)+\sum\limits_{n=1}^\infty l(J_n)\geq M_u(A\cap I)+M_u(A\cap J).$$
It follows that $(A\cap I)\cup(A\cap J)$ is outer measurable and that
$$M_u((A\cap I)\cup(A\cap J))=M_u(A\cap I)+M_u(A\cap J).$$
\end{proof}
\begin{col}
Let $A\subset\mathcal{R}$ be outer measurable and let $\{J_n\}_{n=1}^N$ be pairwise disjoint intervals in $\mathcal{R}$. Then $A\cap \left(\bigcup\limits_{n=1}^NJ_n\right)$ is outer measurable and
$$M_u\left(A\cap \left(\bigcup\limits_{n=1}^NJ_n\right)\right)=\sum\limits_{n=1}^NM_u(A\cap J_n).$$
\end{col}
\begin{propn}
Let $A\subset\mathcal{R}$ be outer measurable and let $\{J_n\}_{n=1}^\infty$ be a sequence of pairwise disjoint intervals in $\mathcal{R}$ with  $\lim\limits_{n\to \infty}l(J_n)=0$. Then $A\cap \left(\bigcup\limits_{n=1}^\infty J_n\right)$ is outer measurable and
$$M_u\left(A\cap \left(\bigcup\limits_{n=1}^\infty J_n\right)\right)=\sum\limits_{n=1}^\infty M_u(A\cap J_n).$$
\end{propn}
\begin{proof}
Since $J_n\cap A\subseteq J_n$, we have that $M_u(A\cap J_n)\leq l(J_n)$. Thus, $\lim\limits_{n\to\infty} M_u(A\cap J_n)=0$ and hence $\sum\limits_{n=1}^\infty M_u(A\cap J_n)$ converges, by Corollary \ref{corinfinitesum}. Let $\{J_{n,m}^k\}_{m=1}^\infty$ be a sequence that converges to $A\cap J_n$. The cover $\bigcup\limits_{n=1}^\infty\{J_{n,m}^k\}$ satisfies the following:
$$\lim_{k\to\infty}\sum\limits_{X\in \bigcup\limits_{n=1}^\infty\{J_{n,m}^k\}}l(X)=\lim_{k\to\infty}\sum\limits_{n=1}^\infty \sum\limits_{m=1}^\infty l(J_{n,m}^k)=\sum\limits_{n=1}^\infty\left( \lim_{k\to\infty}\sum\limits_{m=1}^\infty l(J_{n,m}^k)\right)=\sum\limits_{n=1}^\infty M_u(A\cap J_n).$$
Suppose now that $\{S_n\}$ is any cover of $A\cap \left(\bigcup\limits_{n=1}^\infty J_n\right)$. Then, for every $N\in\setN$, we have
$$\sum\limits_{n=1}^\infty l(S_n)\geq \sum\limits_{n=1}^N M_u(A\cap J_n)=M_u\left(A\cap \left(\bigcup\limits_{n=1}^NJ_n\right)\right)$$
Hence $\sum\limits_{n=1}^\infty l(S_n)\geq \sum\limits_{n=1}^\infty M_u(A\cap J_n)$.
We conclude that $A\cap \left(\bigcup\limits_{n=1}^\infty J_n\right)$ is outer measurable and $$M_u\left(A\cap \left(\bigcup\limits_{n=1}^\infty J_n\right)\right)=\sum\limits_{n=1}^\infty M_u(A\cap J_n).$$
\end{proof}
\begin{propn}
Let $A\subset\mathcal{R}$ be outer measurable and let $I$ be an interval in $\mathcal{R}$ such that $I\cap A=\emptyset$. Then $A \cup I$ is outer measurable with
$M_u(A\cup I)=M_u(A)+l(I)$.
\end{propn}
\begin{proof}
Let $\{J_n^k\}$ be a sequence that converges to $A$. Without loss of generality, we may assume that $J_n^k=J_n^k\cap I^c$. Then, if we define
$$I_0^k:=I\text{ and }I_n^k:=J_n^k\text{ for each }n\in\setN,$$
we obtain that $A\cup I\subseteq \bigcup\limits_{n=0}^\infty I_n^{k+1}\subseteq \bigcup\limits_{n=0}^\infty I_n^{k}$ and that ${\lim\limits_{k\to\infty}\sum\limits_{n=0}^\infty l(I_n^{k})=M_u(A)+l(I)}$.\\

Let $\{S_n\}$ be a cover of $A\cup I$. We can subdivide $\{S_n\}$ into $\{S_n\cap I\}$ and $\{S_n\cap I^c\}$ covers of $I$ and $A$, respectively. Thus,
\[
    \sum\limits_{n=0}^\infty l(S_n)=\sum\limits_{n=0}^\infty (l(S_n\cap I)+l(S_n\cap I^c))=\sum\limits_{n=0}^\infty l(S_n\cap I)+\sum\limits_{n=0}^\infty l(S_n\cap I^c)\geq l(I)+M_u(A).
\]
We conclude that  $A \cup I$ is outer measurable with
$M_u(A\cup I)=M_u(A)+l(I)$.
\end{proof}
\begin{col}
Let $A\subset\mathcal{R}$ be outer measurable, let $N\ge 2$ in $\N$ be given, and let  $\{I_n\}_{n=1}^N$ be pairwise disjoint intervals in $\mathcal{R}$ such that $A\cap I_n=\emptyset$ for each $n\in\{1,\ldots,N\}$. Then $A\cup\left(\bigcup\limits_{n=1}^N I_n\right)
$ is outer measurable, with
$$M_u\left(A\cup\left(\bigcup\limits_{n=1}^N I_n\right)\right)=M_u(A)+\sum\limits_{n=1}^N l(I_n)$$
\end{col}
\begin{col}
Let $A\subset\mathcal{R}$ be outer measurable, let $N\in\N$ be given, and let  $\{I_n\}_{n=1}^N$ be intervals in $\mathcal{R}$. Then $A\cup\left(\bigcup\limits_{n=1}^NI_n\right)$ is outer measurable.
\end{col}
\begin{proof}
It is enough to see that ${A\cup I=(A\cap I^c)\cup I}$, which is outer measurable.
\end{proof}
\begin{propn}
Let $A\subset\mathcal{R}$ be outer measurable and let  $\{I_n\}_{n=1}^\infty$ be a collection of pairwise disjoint intervals in $\mathcal{R}$ such that $A\cap I_n=\emptyset$ for each $n\in\setN$ and $\lim\limits_{n\to\infty}l(I_n)=0$. Then $A\cup\left(\bigcup\limits_{n=1}^\infty I_n\right)
$ is outer measurable, with
$$M_u\left(A\cup\left(\bigcup\limits_{n=1}^\infty I_n\right)\right)=M_u(A)+\sum\limits_{n=1}^\infty l(I_n).$$
\end{propn}
\begin{proof}
Let $\{J_n^k\}$ be a sequence converging to $A$. Then $\{J_n^k\}\cup\{I_n\}$ covers $A\cup\left(\bigcup\limits_{n=1}^\infty I_n\right)$ and
$$\lim_{k\to\infty}\sum\limits_{X\in \{J_n^k\}\cup\{I_n\}_{n=1}^\infty}l(X)=\lim_{k\to\infty}\sum\limits_{n=1}^\infty l(J_n^k)+\sum\limits_{n=1}^\infty l(I_n)=M_u(A)+\sum\limits_{n=1}^\infty l(I_n).$$
Now, let $\{S_n\}$ be a cover of $\{A\}\cup\{I_n\}_{n=1}^\infty$. It follows that
$$M_u\left(A\cup\left(\bigcup\limits_{n=1}^N I_n\right)\right)=M_u(A)+\sum\limits_{n=1}^N l(I_n)\leq \sum\limits_{n=1}^\infty l(S_n)$$
for every $N\in\setN$, and hence $M_u(A)+\sum\limits_{n=1}^\infty l(I_n)\leq \sum\limits_{n=1}^\infty l(S_n)$.
\end{proof}
\begin{propn}
Let $A\subset\mathcal{R}$ be outer measurable and let  $\{I_n\}_{n=1}^\infty$ be a collection of intervals in $\mathcal{R}$ such that $\lim\limits_{n\to\infty}l(I_n)=0$. Then $A\cup\left(\bigcup\limits_{n=1}^\infty I_n\right)$ is outer measurable.
\end{propn}
\begin{proof}
Without loss of generality, we may assume that $\{I_n\}_{n=1}^\infty$ is a pairwise disjoint collection that's arranged in the order of decreasing length. For each $m\in \mathbb{N}$, there exist $N_m\in\setN$ and some cover $\{S_n^m\}$ of  $A_m:=A\cup\left(\bigcup\limits_{n=1}^{N_m}I_n\right)$ such that
$$\sum\limits_{n=N_m+1}^\infty l(I_n)<d^m\text{ and }\sum\limits_{n=1}^\infty l(S^m_n)-M_u(A_m)<d^m.$$
Then $C_m:=\{S_n^m\}\cup\{I_n\}_{n=N_m+1}^\infty$ is a sequence of covers for $A\cup\left(\bigcup\limits_{n=1}^\infty I_n\right)$ that satisfies
\begin{align*}
    \left|\sum\limits_{X\in C_{m+1}}l(X)-\sum\limits_{X\in C_m} l(X)\right|&=\left|\sum\limits_{n=N_{m+1}+1}^\infty l(I_n)+\sum\limits_{n=1}^\infty l(S^{m+1}_n)-\sum\limits_{n=N_{m}+1}^\infty l(I_n)-\sum\limits_{n=1}^\infty l(S^m_n)\right|\\
    &\leq \left|\sum\limits_{n=N_{m+1}+1}^\infty l(I_n)\right|+\left|\sum\limits_{n=N_{m}+1}^\infty l(I_n)\right|+\left| \sum\limits_{n=1}^\infty l(S^{m+1}_n)-\sum\limits_{n=1}^\infty l(S^m_n)\right|\\
    &<2d^m+\left| \sum\limits_{n=1}^\infty l(S^{m+1}_n)-\sum\limits_{n=1}^\infty l(S^m_n)\right|\\
    &\leq 2d^m+\left|\sum\limits_{n=1}^\infty l(S^{m+1}_n)-M_u(A_{m+1})\right|+\left|M_u(A_{m+1})-M_u(A_{m})\right|\\
    &\hspace{.1in}+\left|M_u(A_{m})-\sum\limits_{n=1}^\infty l(S^m_n)\right|\\
    &<4d^m+\left|M_u(A_{m+1})-M_u(A_{m})\right|\\
    &\leq 4d^m+\sum\limits_{n=N_{m}+1}^\infty l(I_n)\\
    &<5d^m.
\end{align*}
Thus, using Corollary \ref{corseqconv}, we infer that $x:=\lim\limits_{m\to\infty}
\sum\limits_{X\in C_m} l(X)$ exists in $\sR$.
Suppose now that $\{J_n\}$ is a cover of $A\cup\left(\bigcup\limits_{n=1}^\infty I_n\right)$ such that $\sum\limits_{n=1}^\infty l(S_n)<x$. We choose $k$ such that $d^k+\sum\limits_{n=1}^\infty l(S_n)<x$ and $m>k$ such that $d^k+\sum\limits_{n=1}^\infty l(S_n)<\sum\limits_{X\in C_m} l(X)$. It follows that
\begin{align*}
    M_u(A_m)&>\sum\limits_{n=1}^\infty l(S^m_n)-d^m=\sum\limits_{X\in C_m} l(X)-\sum\limits_{n=N_{m}+1}^\infty l(I_n)-d^m\\
    &>\sum\limits_{X\in C_m} l(X)-2d^m>\sum\limits_{X\in C_m} l(X)-d^k\\
     &>\sum\limits_{n=1}^\infty l(S_n),
\end{align*}
which is a contradiction, since $\{S_n\}$ covers $A_m$. We conclude that $A\cup\left(\bigcup\limits_{n=1}^\infty I_n\right)$ is outer measurable.
\end{proof}
\begin{lemma}\label{A-cupInisum}
   Let $A\subset\mathcal{R}$ be outer measurable and let  $\{I_n\}_{n=1}^\infty$ be a sequence of pairwise disjoint intervals in $\mathcal{R}$ such that $\lim\limits_{n\to\infty}l(I_n)=0$. Then the set $A\bkl \bigcup\limits_{n=1}^\infty I_n$
    is outer measurable, and
    $$M_u\left(A\bkl \bigcup\limits_{n=1}^\infty I_n\right)=\lim\limits_{k\to\infty}M_u\left(A\bkl \bigcup\limits_{n=1}^k I_n\right).$$
\end{lemma}
\begin{proof}
   For each $k\in\setN$, let
    $A_k:=A\bkl \bigcup\limits_{n=1}^k I_n$ and let $A_\infty:=A\bkl \bigcup\limits_{n=1}^\infty I_n$. Then, clearly, $A_{k+1}\subseteq A_k$ and hence $M_u(A_{k+1})\leq M_u(A_k)$ for all $k\in\setN$. We note that $A_k=A_{k+1}\cup(A\cap I_{k+1})$, and hence
    $$M_u(A_k)-M_u(A_{k+1})\leq M_u(A\cap I_{k+1})\leq l(I_{k+1}).$$
    Since $0\le M_u(A_k)-M_u(A_{k+1})\leq l(I_{k+1})$ and since $\lim\limits_{k\to\infty} l(I_{k+1})=0$, we obtain that $\lim\limits_{k\to\infty}\left(M_u(A_k)-M_u(A_{k+1})\right)=0$. It follows then from Corollary \ref{corseqconv} that $\lim\limits_{k\to\infty}M_u\left(A_k\right)$ exists in $\sR$. We define
    $$x:=\lim\limits_{k\to\infty}M_u\left(A_k\right)=\lim\limits_{k\to\infty}M_u\left(A\bkl \bigcup\limits_{n=1}^k I_n\right).$$
    Since each $A_k$ is outer measurable, there exists a cover $\{S_n^k\}$ of $A_k$ such that $\sum\limits_{n=1}^\infty l(S_n^k)-M_u(A_k)<d^k$. Now, the sequence $\{S_n^k\}$ is a sequence of covers of $A_\infty$ that satisfies
    $$\lim\limits_{k\to\infty}\sum\limits_{n=1}^\infty l(S_n^k)=x.$$
   It remains to show that $x$ is a lower bound for the sum of lengths of any countable collection of intervals covering $A_\infty$. Assume to the contrary that $\{J_n\}$ is a cover of $A_\infty$ such that
    $\sum\limits_{n=1}^\infty l(J_n)<x$.
    We now take $N\in\N$ such that
    $$\sum\limits_{n=1}^\infty l(J_n)+\sum\limits_{n=N}^\infty l(I_n)<x\leq M_u(A_N)$$
    which yields a contradiction to the fact that $\{J_n\}\cup\{I_n\}_{n=N}^\infty$ covers $A_N$.
\end{proof}
\begin{propn}\label{AcupBisadittive}
    Let $A,B\subset\mathcal{R}$ be outer measurable such that $A \subseteq \bigcup_{n=1}^\infty I_n$ and $B \subseteq (\bigcup_{n=1}^\infty I_n)^c$, where the $I_n$'s are pairwise disjoint intervals in $\sR$ with $\lim\limits_{n\to\infty}l(I_n)=0$. Then $A \cup B$ is outer measurable, and
    $$M_u(A \cup B) = M_u(A) + M_u(B).$$
\end{propn}
\begin{proof}
    First note that, for any $N \in \N$, we have that
\[
    \left(\bigcup_{n=1}^N I_n\right)^c = \bigcup_{n=1}^{N+1} J_n^N
\]
where $J_n^N$ are pairwise disjoint intervals (possibly infinite). The superscript $N$ means that the partition of $\left(\bigcup_{n=1}^N I_n\right)^c$ into intervals depends on $N$. Now, let $\{S_n\}$ be a cover of $A \cup B$. Then
\begin{align*}
    \sum_{n=1}^\infty l(S_n) &= \sum_{n=1}^\infty l\left(S_n \cap \bigcup_{m=1}^N I_m\right) + l\left(S_n \cap \left(\bigcup_{m=1}^N I_m\right)^c\right) \\
    &= \sum_{n=1}^\infty \left( \sum_{m=1}^N l(S_n \cap I_m) + \sum_{m=1}^{N+1} l(S_n \cap J_m^N) \right) \\
    &= \sum_{n=1}^\infty \left( \sum_{m=1}^\infty l(S_n \cap I_m) - \sum_{m=M+1}^\infty l(S_n \cap I_m) + \sum_{m=1}^{N+1} l(S_n \cap J_m^N) \right) \\
    &= \sum_{n=1}^\infty \sum_{m=1}^\infty l(S_n \cap I_m) - \sum_{n=1}^\infty \sum_{m=N+1}^\infty l(S_n \cap I_m) + \sum_{n=1}^\infty \sum_{m=1}^{N+1} l(S_n \cap J_m^N).
\end{align*}
Note that $\{S_n \cap I_m\ |\ n, m \in \N\}$ is a cover for $A$ (an appropriate indexing can be found easily). Also, $\{S_n \cap J_m^N\ |\ n \in \N,\ m = 1,2,\cdots, N+1\}$ is a cover for $B$ since
\[
    B \subseteq \left(\bigcup_{m=1}^\infty I_m \right)^c \subseteq \left(\bigcup_{m=1}^N I_m \right)^c = \bigcup_{m=1}^{N+1} J_m^N
\]
Thus,
\[
    \sum_{n=1}^\infty \sum_{m=1}^\infty l(S_n \cap I_m) \geq M_u(A)
\text{ and } \sum_{n=1}^\infty \sum_{m=1}^{N+1} l(S_n \cap J_m^N) \geq M_u(B),
\]
and hence
\[
    \sum_{n=1}^\infty l(S_n) \geq M_u(A) + M_u(B) - \sum_{n=1}^\infty \sum_{m=N+1}^\infty l(S_n \cap I_m).
\]
It follows that
\[
     \sum_{n=1}^\infty l(S_n) \geq M_u(A) + M_u(B) - \sum_{m=N+1}^\infty \sum_{n=1}^\infty l(S_n \cap I_m).
\]
Taking the limit as $N\to\infty$ yields
\[
     \sum_{n=1}^\infty l(S_n) \geq M_u(A) + M_u(B).
\]
Finally, we check that if $\{L_n^k\}$, $\{T_n^k\}$ converge to $A$ and $B$, respectively, then $\{L_n^k\}\cup\{T_n^k\}$ covers $A\cup B$ and
\[
    \lim_{k \to \infty} \left(\sum_{n=1}^\infty l(L_n^k) + \sum_{n=1}^\infty l(T_n^k)\right) = M_u(A) + M_u(B).
\]
We conclude that $A\cup B$ is outer measurable and $M_u(A\cup B)=M_u(A)+M_u(B)$.
\end{proof}
\begin{theorem}
 Let $A,B\subset\mathcal{R}$ be outer measurable. Then $A\cup B$ is outer measurable.
\end{theorem}
\begin{proof}
Let $\{I_n^k\}$ be a sequence that converges to $A$. We define
$$B_k:=B\cap\bigcap\limits_{n=1}^\infty \left(I_n^k\right)^c.$$
Since each $B_k$ is outer measurable, there exists a cover $\{J_n^k\}$ of $B_k$ such that $\sum\limits_{n=1}^\infty l(J_n^k)-M_u(B_k)<d^k$. We note that, since
$\bigcup\limits_{n=1}^\infty I_n^{k+1}\subseteq\bigcup\limits_{n=1}^\infty I_n^{k}$,
then
$\bigcap\limits_{n=1}^\infty \left(I_n^{k}\right)^c\subseteq\bigcap\limits_{n=1}^\infty \left(I_n^{k+1}\right)^c$
and hence $B_k\subseteq B_{k+1}$. Moreover
\begin{align*}
    B_{k+1}&= B_k\cup(B_{k+1}\bkl B_k)\\
    &=B_k\cup \left(B\cap \left(\bigcup\limits_{n=1}^\infty I_n^{k}\bkl\bigcup\limits_{n=1}^\infty I_n^{k+1}\right)\right)\\
     &=B_k\cup \left(\bigcup\limits_{n=1}^\infty I_n^{k}\cap \left(B\bkl\bigcup\limits_{n=1}^\infty I_n^{k+1}\right)\right)\\
     &\subseteq B_k\cup \left(\bigcup\limits_{n=1}^\infty I_n^{k}\bkl\bigcup\limits_{n=1}^\infty I_n^{k+1}\right).
\end{align*}
Hence $B_{k+1}\bkl B_k$ is outer measurable and $M_u(B_{k+1}\bkl B_k)\leq \sum\limits_{n=1}^\infty l(I_n^{k})-\sum\limits_{n=1}^\infty l(I_n^{k+1})\to 0$ as $k\to\infty$. It follows that the sequence $M_u(B_k)$ is Cauchy, and it converges in $\mathcal{R}$. That is,
$$\lim_{k\to\infty}\sum\limits_{n=1}^\infty l(J_n^{k})=\lim_{k\to\infty}M_u(B_k)\in \mathcal{R}.$$
We define $B_\infty:=\bigcup\limits_{n=1}^\infty B_k$. Let $\{S_n\}$ be a cover of $B_\infty$. It follows that $\{S_n\}$ covers $B_k$ for each $k\in\setN$; thus, $M_u(B_k)\leq\sum\limits_{n=1}^\infty l(S_n)$ for each $k\in\setN$ and hence $x:=\lim\limits_{k\to\infty} M_u(B_k)\leq\sum\limits_{n=1}^\infty l(S_n)$.

Since $B_{k+1}\bkl B_k$ is outer measurable for each $k\in\setN$ and $\lim\limits_{k\to\infty}M_u(B_{k+1}\bkl B_k)=0$, there exists $\{R_n^k\}$ cover of $B_{k+1}\bkl B_k$ such that $\sum\limits_{n=1}^\infty l(R_n^k)-M_u(B_{k+1}\bkl B_k)<d^k$. It follows that $\{J_n^k\}\cup\bigcup\limits_{m=k}^\infty\{R_n^m\}$ covers $B_\infty=B_k\cup\left(\bigcup\limits_{m=k}^\infty( B_{m+1}\bkl B_m)\right)$ and
$\lim\limits_{k\to\infty}\sum\limits_{n=1}^\infty l(J_n^k)+\sum\limits_{m=k}^\infty\sum\limits_{n=1}^\infty l(R_n^m)=x$.
We conclude that $B_\infty$ is outer measurable and has outer measure $x$.

Now, $\{J_n^k\}\cup\{I_n^k\}$ covers $A\cup B_k$ and
$$\lim_{k\to\infty}\left(\sum\limits_{n=1}^\infty l(J_n^k)+\sum\limits_{n=1}^\infty l(I_n^k)\right)=M_u(A)+x.$$
Finally, if $\{T_n\}$ covers $A\cup B$, then it does also cover $A\cup B_k$, which is outer measurable by Proposition $\ref{AcupBisadittive}$. Thus,
$$M_u(A)+M_u(B_k)=M_u(A\cup B_k)\leq \sum\limits_{n=1}^\infty l(S_n)$$
and hence
$$M_u(A)+x=M_u(A)+ \lim_{k\to \infty} M_u(B_k)\leq \sum\limits_{n=1}^\infty l(S_n).$$
We conclude that $A\cup B$ is outer measurable.
\end{proof}
\subsection{The problem with the outer measure.}
From the last result, one can see that the outer measure may be used to introduce a strong enough definition of measurability to serve as a replacement for the S-measure in the Levi-Civita field $\sR$. However, similar to the situation in classical real analysis, the new concept fails in extreme cases when sets are too fuzzy or too tangled together as to pass the measurability test but fail for intersections and additivity. We give an example (Example \ref{examplenoint} below) that illustrates both cases. But first we prove the following result which will be used in that example.
\begin{propn}
Let $I\subset\sR$ be an interval, $X\subseteq I$ a dense subset of $I$, and $\{S_n\}$ a cover of $X$. Then $\sum\limits_{n=1}^\infty l(S_n)\geq l(I)$.
\end{propn}
\begin{proof}
Let $k\in\Q$ be given. For every $t\in\sR$, we define the $k$-th truncation of $t$ as follows:
$$t^{(k)}=\sum\limits_{q\leq k}t[q]d^q.$$

Now, since $X$ is dense in $I$, then, for every $s\in I$, there exists $x\in X$ such that $x=_k s$. Let $N\in\N$ to be such that for every $n>N$, $l(S_n)\ll d^k$. It follows that if $s,t\in S_n$, then $s =_k t$. Consider now the finite sub-collection of intervals $\{S_n\}_{n=1}^N=\{(a_n,b_n)\}_{n=1}^N$ arranged such that
$$a^{(k)}\leq a_1^{(k)}\leq b_1^{(k)}\leq a_2^{(k)}\leq b_2^{(k)}\leq...\leq a_N^{(k)}\leq b_N^{(k)}\leq b^{(k)}.$$
Suppose now that $a^{(k)}< a_1^{(k)}$. Since the interval $A:=\left (a^{(k)}, a_1^{(k)}\right)$ is uncountable but the set $\{x_n^{(m)}\}_{n=N+1}^\infty$ where $x_n^{(k)}\in S_n$ is countable, there exists $t\in A$ such that $t^{(k)}\neq x_n^{(k)}$ for all $n\in N$. It follows that there exists $x\in X$ such that $x=_kt^{(k)}$. However, this is a contradiction, because $x\notin S_n$ for all $n\in \mathbb{N}$. We conclude that $a^{(k)}=a_1^{(k)}$. Using an identical argument one shows that $b_n^{(k)}=a_{n+1}^{(k)}$. It follows that
$$\sum\limits_{n=1}^\infty l(S_n)=_k\sum\limits_{n=1}^\infty \left(b_n^{(k)}-a_n^{(k)}\right)=_k b^{(k)}-a^{(k)}=_k b-a=_k l(I).$$
Since this is true for any arbitrary $k$, one has that
$\sum\limits_{n=1}^\infty l(S_n)=l(I)$.
\end{proof}
\begin{col}
Let $X$ be a dense subset of an interval $I$ in $\sR$. Then $X$ is outer measurable and $M_u(X)=l(I)$.
\end{col}
\begin{eg}\label{examplenoint}
Consider the interval $I=[0,1]$ and the sets
\begin{align*}
    T&:=\left\{x\in[0,1]\mid \exists N\in \mathbb{N}; \forall q>N, x[q]=0\right\}\\
    S&:=\left\{x\in[0,1]\mid \forall N\in \mathbb{N}; \exists q>N, x[q]\neq 0\right\}
\end{align*}
Clearly $T,S$ are both dense in $[0,1]$ and $T\cap S=\emptyset$. Consider now the set
$$C:=\bigcup\limits_{n=2}^\infty \left(d^{(n-1)/n},2d^{(n-1)/n}\right)$$
It is not difficult to show that $C$ is not outer measurable; on the other hand, $T\cup C$ and $J\cup C$ are both dense in $[0,1]$, and hence they are outer measurable. Moreover,
$$(T\cup C)\cap (S\cup C)= C\cup (T\cap S)=C .$$
Thus, the intersection of two outer measurable sets is not necessarily outer measurable.
\end{eg}
\section{A Lebesgue-like measure.}
With the results from the previous section, we are ready to define a new measure and a new family of measurable sets in $\sR$, making use of the outer measure on $\sR$, similarly to how the Lebesgue measure of real analysis is defined in terms of the outer measure on $\R$.
\begin{defn}
   Let $A\subset\mathcal{R}$ be an outer measurable set. Then we say that $A$ is L-measurable if for every other outer measurable set $B\subset\mathcal{R}$ both $A\cap B$ and $A^c\cap B$ are outer measurable and
    $$M_u(B)=M_u(A\cap B)+M_u(A^c\cap B).$$
    In this case, we define the L-measure of $A$ to be $M(A):=M_u(A)$. The family of L-measurable sets in $\mathcal{R}$ will be denoted by $\lM$.
\end{defn}
\begin{propn}
   Let $A,B\in\lM$ be given. Then $A\cap B,A\cup B,A\cup B^c\in\lM$.
\end{propn}
\begin{proof}
Let $X\subset\mathcal{R}$ be outer measurable. Then, by definition, the sets
$A\cap X$, $B\cap X$, $A^c\cap X$ and $B^c\cap X$
    are all outer measurable. It follows that the sets
$A\cap B\cap X$, $(A\cup B)\cap X$, $A^c\cap B^c\cap X$, $(A^c\cup B^c)\cap X$, $X\cap A\cap B^c$, $X\cap A^c\cap B$, $(A^c\cup B)\cap X$, and $(A\cup B^c)\cap X$
are all outer measurable; and so are the sets $X\cap (A^c\cup B^c)\cap A$, $X\cap (A^c\cup B^c)\cap A^c$, $X\cap (A\cup B)\cap A$, $X\cap (A\cup B)\cap A^c$, $X\cap (A^c\cup B)\cap A$, and $X\cap (A^c\cup B)\cap A^c$.

Now, we simply check that
    \begin{align*}
        M_u(X)&=M_u(X\cap A)+M_u(X\cap A^c)\\
        &=M_u(X\cap A\cap B)+M_u(X\cap A\cap B^c)+M_u(X\cap A^c)\\
        &=M_u(X\cap A\cap B)+M_u(X\cap (A^c\cup B^c)\cap A)+M_u(X\cap (A^c\cup B^c)\cap A^c)\\
        &=M_u(X\cap A\cap B)+M_u(X\cap (A^c\cup B^c))\\
        &=M_u(X\cap (A\cap B))+M_u(X\cap (A\cap B)^c);
    \end{align*}
    \begin{align*}
        M_u(X)&=M_u(X\cap A^c)+M_u(X\cap A)\\
        &=M_u(X\cap A^c\cap B^c)+M_u(X\cap A^c\cap B)+M_u(X\cap A)\\
        &=M_u(X\cap A^c\cap B^c)+M_u(X\cap (A\cup B)\cap A^c)+M_u(X\cap (A\cup B)\cap A)\\
        &=M_u(X\cap A^c\cap B^c)+M_u(X\cap (A\cup B))\\
        &=M_u(X\cap (A\cup B)^c)+M_u(X\cap (A\cup B));
    \end{align*}
    and
    \begin{align*}
        M_u(X)&=M_u(X\cap A)+M_u(X\cap A^c)\\
        &=M_u(X\cap A\cap B^c)+M_u(X\cap A\cap B)+M_u(X\cap A^c)\\
        &=M_u(X\cap A\cap B^c)+M_u(X\cap (A^c\cup B)\cap A)+M_u(X\cap (A^c\cup B)\cap A^c)\\
        &=M_u(X\cap A\cap B^c)+M_u(X\cap (A^c\cup B))\\
        &=M_u(X\cap (A\cap B^c))+M_u(X\cap (A\cap B^c)^c).
    \end{align*}
Thus, $A\cup B$, $A\cap B$ and $A\cap B^c$ are L-measurable.

\end{proof}
The family $\lM$ of L-measurable sets naturally inherits some of the key properties of the Lebesgue measure in $\R$.

\begin{propn}
    Let $A,B\in\lM$ be given. Then
    $$M(A\cup B)=M(A)+M(B)-M(A\cap B).$$
\end{propn}
\begin{proof}
    We already know that $A\cup B,A\cap B\in\lM$. It follows that
    \begin{align*}
        M(A\cup B)&=M_u(A\cup B)\\
        &=M_u((A\cup B)\cap A)+M_u((A\cup B)\cap A^c))\\
        &=M_u(A)+M_u(B\cap A^c)\\
        &=M_u(A)+M_u(B)-M_u(A\cap B)\\
        &=M(A)+M(B)-M(A\cap B).
    \end{align*}
\end{proof}
This family $\lM$ also proves to be a direct improvement over the S-measure by strictly expanding the family of S-measurable sets.
\begin{propn}
Let $\{J_n\}$ be a sequence of pairwise disjoint intervals in $\mathcal{R}$ such that $\lim\limits_{n\to\infty}l(J_n)=0$. Then $\bigcup\limits_{n=1}^{\infty}J_n$ is L-measurable, and
 $$M\left(\bigcup\limits_{n=1}^{\infty}J_n\right)=\sum\limits_{n=1}^{\infty}l(J_n).$$
\end{propn}
\begin{proof}
    The result follows directly from Lemma \ref{A-cupInisum} and Proposition \ref{AcupBisadittive} and from the fact that
    $$A=\left(A\cap\bigcup\limits_{n=1}^{\infty}J_n \right)\cup\left(A\bkl\bigcup\limits_{n=1}^{\infty}J_n\right).$$
\end{proof}
\begin{propn}
    Let $C\subset\mathcal{R}$ be outer measurable with $M_u(C)=0$. Then $C$ is L-measurable with M(C)=0.
\end{propn}
\begin{proof}
    Let $A\subset\mathcal{R}$ be outer measurable. Then, using Proposition \cref{0measuresets}, we have that
    $$M_u(A)=M_u(A\bkl C)=M_u(A\bkl C)+0=M_u(A\bkl C)+M_u(A\cap C)=M_u(A\cap C)+M_u(A\cap C^c). $$
    Thus, $C$ is L-measurable, with L-measure $M(C)=M_u(C)=0$.
\end{proof}
\begin{col}
     Let $A\subset\mathcal{R}$ be  S-measurable. Then $A$ is L-measurable and $M(A)=M_s(A)$.
\end{col}
One of the key results in probability theory is that of the continuity of the probability function. Despite not being able to get the full result due to the intrinsic characteristics of the set, we manage to get very close to it.
\begin{lemma}\label{lemmaprobunion0}
    For each $n\in\setN$ let $A_n\subset\mathcal{R}$ be L-measurable, with $\lim\limits_{n\to\infty}M(A_n)=0$, and let $X\subset\mathcal{R}$ be outer measurable. Then
\[
    \lim\limits_{N\to\infty}M_u\left(X\cap\bigcup\limits_{n=1}^NA_n\right)=M_u\left(X\cap\bigcup\limits_{n=1}^\infty A_n\right)
\]
\end{lemma}
\begin{proof}
    For each $N\in\N$ let $X_N:=X\cap\bigcup\limits_{n=1}^NA_n$, and let $X_\infty:=X\cap\bigcup\limits_{n=1}^\infty A_n$. Then we have that
    \[
M_u(X_{N+1})\leq M_u(X_N)+M_u(A_{N+1})\text{ for all }N\in\N.
    \]
  Since the outer measures of $A_n$ form a null sequence, then the sequence $M_u(X_N)$ is Cauchy and therefore convergent in $\mathcal{R}$. We define
    \[
        t:=\lim_{N\to\infty}M_u(X_N).
    \]
    Given that $X_m$ and $A_k$ are outer measurable for each $m\in\N$ and for each $k\in\N$, we can find covers $\{J_n^m\}$ and $\{I_n^k\}$ of $X_m$ and $A_k$, respectively, such that
    \[
        \sum\limits_{n=1}^{\infty}l(J^m_n)-M_u(X_m)<d^m\text{ and }\sum\limits_{n=1}^{\infty}l(I^k_n)-M_u(A_k)<d^k.
    \]
    We define $\{S_n^k\}:=\{J_n^k\}\cup\bigcup\limits_{m=k+1}^\infty \{I_n^m\}$. The, clearly, $\{S_n^k\}$ is a covering of $X_\infty$ and
    \[
        \lim_{k\to\infty}\sum\limits_{n=1}^{\infty}l(S_n^k)=
        \lim_{k\to\infty}\left(\sum\limits_{n=1}^{\infty}l(J_n^k)+\sum\limits_{m=k+1}^{\infty}\sum\limits_{n=1}^{\infty}l(I_n^m)\right)=t.
    \]
    Finally, if $\{P_n\}$ covers $X_\infty$, then it covers $X_N$ for all $N\in\N$. Thus,
   $M_u(X_N)\leq \sum\limits_{n=1}^{\infty}l(P_n)$ for all $N\in\N$ and hence
    \[
        \lim_{N\to\infty}M_u(X_N)=t\leq \sum\limits_{n=1}^{\infty}l(P_n).
    \]
    We conclude that $X_\infty$ is outer measurable and $t=M_u(X_\infty)$; that is, $\lim_{N\to\infty}M_u(X_N)=M_u(X_\infty)$ or
    \[
    \lim\limits_{N\to\infty}M_u\left(X\cap\bigcup\limits_{n=1}^NA_n\right)=M_u\left(X\cap\bigcup\limits_{n=1}^\infty A_n\right).
\]
\end{proof}
\begin{col}\label{xcaplimunion}
 For each $n\in\setN$ let $A_n\subset\mathcal{R}$ be L-measurable such that $\lim\limits_{N\to\infty}M\left(\bigcup\limits_{n=1}^N A_n\right)$ exists in $\sR$, and let $X\subset\mathcal{R}$ be outer measurable. Then
\[
    \lim\limits_{N\to\infty}M_u\left(X\cap\bigcup\limits_{n=1}^NA_n\right)=M_u\left(X\cap\bigcup\limits_{n=1}^\infty A_n\right)
\]
\end{col}
\begin{proof}
    Set $A_0=\emptyset$. Then, since $\lim\limits_{N\to\infty}M\left(\bigcup\limits_{n=0}^N A_n\right)=\lim\limits_{N\to\infty}M\left(\bigcup\limits_{n=1}^N A_n\right)$ exists in $\sR$, we have that
    \[
        M\left(\bigcup\limits_{n=0}^{N+1} A_n\right)-M\left(\bigcup\limits_{n=0}^N A_n\right)=M\left(\bigcup\limits_{n=0}^{N+1} A_n\bkl \bigcup\limits_{n=0}^N A_n\right)\underset{N\to\infty}{\to} 0.
    \]
    For each $N\in\N$, let $B_N=\bigcup\limits_{n=0}^{N} A_n\bkl \bigcup\limits_{n=0}^{N-1} A_n$. Then the new sequence satisfies the following
    \[
    \bigcup\limits_{n=1}^NB_n=\bigcup\limits_{n=1}^N A_n, \bigcup\limits_{n=1}^\infty B_n=\bigcup\limits_{n=1}^\infty  A_n\text{ and } \lim\limits_{n\to\infty}M(B_n)=0.
    \]
    The result follows from Lemma \ref{lemmaprobunion0}.
\end{proof}
\begin{col}
    For each $n\in\setN$ let $A_n\subset\mathcal{R}$ be L-measurable such that $\lim\limits_{N\to\infty}M\left(\bigcup\limits_{n=1}^N A_n\right)$ exists in $\sR$. Then
\[
    \lim\limits_{N\to\infty}M_u\left(\bigcup\limits_{n=1}^NA_n\right)=M_u\left(\bigcup\limits_{n=1}^\infty A_n\right).
\]
\end{col}
\begin{proof}
    Without loss of generality, we may assume that $\lim\limits_{m\to\infty}M(A_m)=0$. Then, for each $m\in\N$ we can find a cover $\{J_n^m\}$ of $A_m$ such that $\lim\limits_{m\to\infty}\left(\sum\limits_{n=1}^\infty l(J_n^m)\right)=0$. Thus, $X:=\bigcup\limits_{m,n=1}^\infty J_n^m$ is outer measurable and the result follows from Corollary \ref{xcaplimunion}.
\end{proof}
\begin{lemma}\label{limintersection}
     For each $n\in\setN$ let $A_n\subset\mathcal{R}$ be L-measurable such that $\lim\limits_{N\to\infty}M\left(\bigcap\limits_{n=1}^N A_n\right)$ exists in $\sR$, and let $X\subset\mathcal{R}$ be outer measurable. Then $X\cap\bigcap\limits_{n=1}^\infty A_n$ is outer measurable, and
\[
    \lim\limits_{N\to\infty}M_u\left(X\cap\bigcap\limits_{n=1}^NA_n\right)=M_u\left(X\cap\bigcap\limits_{n=1}^\infty A_n\right).
\]
\end{lemma}
\begin{proof}
    For each $N\in\N$, define $B_N:=\bigcap\limits_{n=1}^NA_n$ and $X_N:=X\cap B_N$, and define $B_\infty :=\bigcap\limits_{n=1}^\infty A_n$ and $X_\infty :=X\cap B_\infty$. Since $\lim\limits_{N\to\infty}M(B_N)=0$, it follows that
    \[
        M_u(B_N\bkl B_{N+1})=M_u(B_N)-M_u(B_{N+1})\to 0\text{ as }N\to\infty.
    \]
    We have that $X_{N+1}\subseteq X_N\subseteq X_{N+1}\cup(B_N\bkl B_{N+1})$, and hence
    \[
        M_u(X_N)\leq M_u(X_{N+1})+M_u(B_N\bkl B_{N+1}).
    \]
   It follows that the sequence $(X_n)_{n\in\N}$ is Cauchy and therefore convergent in $\sR$. Let $t=\lim\limits_{N\to\infty}X_N$.

    For each $k\in\N$, let $\{J_n^k\}$ be a cover by intervals of $X_k$ such that $\sum\limits_{n=1}^\infty l(J_n^k)-M_u(X_k)<d^k$. Thus, $\{J_n^k\}$ is a sequence of covers of $X_\infty$ satisfying $\lim\limits_{k\to\infty}\sum\limits_{n=1}^\infty l(J_n^k)=t$.

    It remains to show that $t\le\sum\limits_{n=1}^\infty l(I_n)$ for any cover $\{I_n\}$ of $X_\infty$. Suppose, to the contrary, that there exists a cover $\{I_n\}$ of $X_\infty$ such that $\sum\limits_{n=1}^\infty l(I_n)<t$. Since $B_m\bkl B_{m+1}\in\lM$ for each $m\in\N$ and since $\lim\limits_{m\to\infty}M(B_m\bkl B_{m+1})=0$, then there exists some $N\in\N$ such that $\sum\limits_{n=1}^\infty l(I_n)+\sum\limits_{m=N}^\infty M(B_m\bkl B_{m+1})<t$. Thus, we can find covers $\{J_n^m\}$ of $B_m\bkl B_{m+1}$ for each $m\ge N$ so that $\sum\limits_{n=1}^\infty l(I_n)+\sum\limits_{m=N}^\infty \sum\limits_{n=1}^\infty l(J_n^m)<t$, which is a contradiction to the fact that $\{I_n\}\cup\bigcup_{m=N}^\infty\{J_n^m\}$ covers $X_N$.

       We conclude that $X_\infty=X\cap\bigcap\limits_{n=1}^\infty A_n$ is outer measurable and $t:=\lim\limits_{N\to\infty}M_u\left(X_N\right)=M_u\left(X_\infty\right)$. That is,
    \[
    \lim\limits_{N\to\infty}M_u\left(X\cap\bigcap\limits_{n=1}^NA_n\right)=M_u\left(X\cap\bigcap\limits_{n=1}^\infty A_n\right).
\]
\end{proof}
\begin{theorem}
     For each $n\in\setN$ let $A_n\subset\mathcal{R}$ be L-measurable and such that $\lim\limits_{N\to\infty}M\left(\bigcup\limits_{n=1}^N A_n\right)$ exists in $\sR$. Then, $\bigcup\limits_{n=1}^\infty A_n$ is L-measurable.
\end{theorem}
\begin{proof}
    We already know that $\bigcup\limits_{n=1}^\infty A_n$ is outer measurable. Now let $X\subset\mathcal{R}$ be outer measurable and let $\{I_n\}$ be a cover of $X$. Since $\bigcup\limits_{n=1}^N A_n$ and $\bigcup\limits_{n=1}^\infty I_n\bkl \bigcup\limits_{n=1}^N A_n$ are L-measurable for each $N\in\N$ and since $\lim\limits_{N\to\infty}M\left(\bigcup\limits_{n=1}^N A_n\right)$ and $\lim\limits_{N\to\infty}M\left(\bigcup\limits_{n=1}^\infty I_n\setminus \bigcup\limits_{n=1}^N A_n\right)$ both exist in $\sR$, we obtain that
    \begin{align*}
        M_u(X)&=\lim\limits_{N\to\infty}M_u(X)\\
        &=\lim\limits_{N\to\infty}\left[M_u\left(X\cap\bigcup\limits_{n=1}^N A_n\right)+M_u\left(X\bkl\bigcup\limits_{n=1}^N A_n\right)\right]\\
        &=\lim\limits_{N\to\infty}\left[M_u\left(X\cap\bigcup\limits_{n=1}^N A_n\right)+M_u\left(X\cap\left(\bigcup\limits_{n=1}^\infty I_n\bkl\bigcup\limits_{n=1}^N A_n\right)\right)\right]\\
       &=\lim\limits_{N\to\infty}M_u\left(X\cap\bigcup\limits_{n=1}^N A_n\right)+\lim\limits_{N\to\infty}M_u\left(X\cap\left(\bigcup\limits_{n=1}^\infty I_n\bkl\bigcup\limits_{n=1}^N A_n\right)\right)\\
        &=M_u\left(X\cap\bigcup\limits_{n=1}^\infty A_n\right)+M_u\left(X\cap\left(\bigcup\limits_{n=1}^\infty I_n\bkl\bigcup\limits_{n=1}^\infty A_n\right)\right)\\
  &=M_u\left(X\cap\bigcup\limits_{n=1}^\infty A_n\right)+M_u\left(X\cap\left(\bigcup\limits_{n=1}^\infty A_n\right)^c\right).
    \end{align*}
    Thus, $\bigcup\limits_{n=1}^\infty A_n$ is L-measurable.
\end{proof}
\begin{theorem}
     For each $n\in\setN$ let $A_n\subset\mathcal{R}$ be L-measurable and such that $\lim\limits_{N\to\infty}M\left(\bigcap\limits_{n=1}^N A_n\right)$ exists in $\sR$. Then, $\bigcap\limits_{n=1}^\infty A_n$ is L-measurable.
\end{theorem}
\begin{proof}
  We already know that $\bigcap\limits_{n=1}^\infty A_n$ is outer measurable. Now let $X\subset\mathcal{R}$ be outer measurable and let $\{I_n\}$ be a cover of $X$. Since $\bigcap\limits_{n=1}^N A_n$ and $\bigcup\limits_{n=1}^\infty I_n\bkl \bigcap\limits_{n=1}^N A_n$ are L-measurable for each $N\in\N$ and since $\lim\limits_{N\to\infty}M\left(\bigcap\limits_{n=1}^N A_n\right)$ and $\lim\limits_{N\to\infty}M\left(\bigcup\limits_{n=1}^\infty I_n\setminus \bigcap\limits_{n=1}^N A_n\right)$ both exist in $\sR$, we obtain that
    \begin{align*}
        M_u(X)&=\lim\limits_{N\to\infty}M_u(X)\\
        &=\lim\limits_{N\to\infty}\left[M_u\left(X\cap\bigcap\limits_{n=1}^N A_n\right)+M_u\left(X\bkl\bigcap\limits_{n=1}^N A_n\right)\right]\\
        &=\lim\limits_{N\to\infty}\left[M_u\left(X\cap\bigcap\limits_{n=1}^N A_n\right)+M_u\left(X\cap\left(\bigcup\limits_{n=1}^\infty I_n\bkl\bigcap\limits_{n=1}^N A_n\right)\right)\right]\\
       &=\lim\limits_{N\to\infty}M_u\left(X\cap\bigcap\limits_{n=1}^N A_n\right)+\lim\limits_{N\to\infty}M_u\left(X\cap\left(\bigcup\limits_{n=1}^\infty I_n\bkl\bigcap\limits_{n=1}^N A_n\right)\right)\\
        &=M_u\left(X\cap\bigcap\limits_{n=1}^\infty A_n\right)+M_u\left(X\cap\left(\bigcup\limits_{n=1}^\infty I_n\bkl\bigcap\limits_{n=1}^\infty A_n\right)\right)\\
  &=M_u\left(X\cap\bigcap\limits_{n=1}^\infty A_n\right)+M_u\left(X\cap\left(\bigcap\limits_{n=1}^\infty A_n\right)^c\right).
    \end{align*}
    Thus, $\bigcap\limits_{n=1}^\infty A_n$ is L-measurable.
\end{proof}
The converse of the results above
do not hold in general due to the extremely strong criteria for convergence in $\sR$.
\begin{eg}
    For each $n\in\N$, let $A_n=(d^{1/n},1/n)$. One can easily check that
$\bigcap\limits_{n=1}^\infty A_n=\emptyset$
    which is L-measurable and has L-measure zero. However,
    \[
    \lim\limits_{N\to\infty}M\left(\bigcap\limits_{n=1}^N A_n\right)=\lim\limits_{N\to\infty}M\left(\bigcap\limits_{n=1}^N \left(d^{1/n},1/n\right)\right)=\lim\limits_{N\to\infty}M\left(d^{1/N},1/N\right
    )=\lim\limits_{N\to\infty}\left(\dfrac{1}{N}-d^{1/N}\right)
    \]
     does not exist in $\sR$.

    Similarly, if, for each $n\in\N$, we let $B_n=[0,1-1/n]\cup[1-d^{1/n},1]$ then it's easy to check that
$\bigcup\limits_{n=1}^\infty B_n=[0,1]$ which is L-measurable and of L-measure 1. However,
    \begin{align*}
    \lim\limits_{N\to\infty}M\left(\bigcup\limits_{n=1}^N B_n\right)&=\lim\limits_{N\to\infty}M\left(\bigcup\limits_{n=1}^N \left([0,1-1/n]\cup[1-d^{1/n},1]\right)\right)\\
    &=\lim\limits_{N\to\infty}M\left(\left([0,1-1/N]\cup[1-d^{1/N},1]\right)\right)\\
    & =\lim\limits_{N\to\infty}\left(1-\dfrac{1}{N}+d^{1/N}\right)
    \end{align*}
     does not exist in $\sR$.
\end{eg}


\end{document}